\documentclass{amsart}
\usepackage{mathrsfs}
\usepackage{subfigure}
\usepackage{latexsym,amsfonts,amssymb}
\usepackage{graphicx,epsfig,subfigure}
\usepackage{psfrag}
\usepackage{enumerate}
\usepackage{colortbl}
\usepackage{verbatim}
\usepackage{hyperref} 

\newtheorem{teo}{Theorem}[section]
\newtheorem{cor}[teo]{Corollary}
\newtheorem{lem}[teo]{Lemma}
\newtheorem{conj}[teo]{Conjecture}
\newtheorem{pro}[teo]{Proposition}

{\theoremstyle{definition}
\newtheorem{defi}[teo]{Definition}
\newtheorem{rem}[teo]{Remark}
\newtheorem{ex}[teo]{Example}}
\newtheorem*{balreira}{\cite[Theorem 2.3]{B}}

\newcommand{\R}{\mathbb R}
\newcommand{\C}{\mathbb C}
\newcommand{\N}{\mathbb N}
\newcommand{\K}{\mathbb K}

\newcommand{\spec}{\operatorname{{Spec}}}

\newcommand{\codim}{\operatorname{{codim}}}
\newcommand{\vazio}{\varnothing}
\newcommand{\eps}{\varepsilon}
\newcommand{\F}{\mathcal{F}}
\newcommand{\Imm}{\operatorname{{Im}}}

\def\Sing{{\rm Sing}}

\newcommand{\cum}{i_1}

\newcommand{\cumk}{i_k}

\newcommand{\cki}[1]{i_{#1}}
\newcommand{\cset}{\{i_1,\ldots,i_{n-2}\}}

\newcommand{\ckif}[1]{F_{i_{#1}}}

\title[On global invertibility of semi-algebraic local diffeomorphisms]{On global invertibility of semi-algebraic local diffeomorphisms}

\author[F. Braun, L.R.G. Dias \MakeLowercase{and} J. Venato-Santos]
{Francisco Braun$^*$, Luis Renato Gon\c{c}alves Dias$^\dagger$ \MakeLowercase{and} Jean Venato-Santos$^\ddagger$}

\address{$^*$ Departamento de Matem\'{a}tica, Universidade Federal de S\~ao Carlos, 
13565-905 S\~ao Carlos, S\~ao Paulo, Brazil} 
\email{franciscobraun@dm.ufscar.br}

\address{$^{\dagger,\ddagger}$ Faculdade de Matem\'{a}tica, Universidade Federal de Uberl\^{a}ndia, 
	38408-100 Uber\-l\^{a}n\-dia, Minas Gerais, Brazil} 
\email{lrgdias@ufu.br}
\email{jvenatos@ufu.br}



\thanks{$^*$ Partial support provided by the grants 2017/00136-0 and 2019/07316-0 of the S\~ao Paulo Research Foundation (FAPESP)}
\thanks{$^\dagger$ Partial support provided by the Fapemig-Brazil Grant APQ-00431-14 and the CNPq-Brazil grants 401251/2016-0 and 304163/2017-1}
\thanks{$^\ddagger$ Partial support provided by the Fapemig-Brazil Grant APQ-00595-14 and the CNPq-Brazil Grant 446956/2014-7}

\begin{document}

\begin{abstract}
In this partly expository paper we discuss conditions for the global injectivity of $C^2$ semi-algebraic local diffeomorphisms $f\colon \R^n \to \R^n$. 
In case $n>2$, we consider the foliations of $\R^n$ defined by the level sets of each $n-2$ projections of $f$, i.e., the maps $\R^n \to \R^{n-2}$ obtained by deleting two coordinate functions of $f$. 
It is known that if the set of non-proper points of $f$ has codimension greater than or equal to $2$ and the leaves of the above-defined foliations are simply connected, then $f$ is bijective. 
In this work we relate this simply connectedness with the notion of locally trivial fibrations. 
Then some computable regularity conditions at infinity ensuring such simply connectedness are presented. 
Further, we provide an equivalent statement of the Jacobian conjecture by using fibrations. 
By means of examples we prove that the results presented here are different from a previous result based on a spectral hypothesis.
Our considerations are also applied to discuss the behaviour of some conditions when $f$ is composed with linear isomorphisms: this is relevant due to some misunderstandings appearing in the literature. 
\end{abstract}

\maketitle

\section{Introduction}

We present conditions in order that a $C^k$ local diffeomorphism $f\colon\R^n\to\R^n$ is a global diffeomorphism. This problem and many such conditions appear in different areas of Mathematics, see for instance \cite{bcw,essen,P} and their references. 
Here we discuss conditions related to Algebraic Geometry and Singularity Theory.  

The first aim of this work is to consider global injectivity results for a class of semi-algebraic local diffeomorphisms $f\colon\R^n\to\R^n$, for $n>2$. 
These are related with Conjecture \ref{JeCo} below. 

We begin introducing some necessary terminology. Let $X$ and $Y$ be locally compact spaces. 
We say that a continuous map $f\colon X\to Y$ is \emph{proper at} $y\in Y$ if there exists a neighborhood $U$
of $y$ such that the set $f^{-1}\left(\overline{U}\right)$ is
compact. 
We denote by $S_f$ the set of points where $f$
is not proper. 
In the case of semi-algebraic maps, it is known that $S_f$ is a semi-algebraic set and hence it makes sense to consider its dimension. 

\begin{conj}[Jelonek's real Jacobian conjecture \cite{J}]\label{JeCo}
	Let $f\colon\R^n\to\R^n$ be a polynomial map with nowhere zero Jacobian determinant. 
	If $\codim(S_f)\geq 2$, then $f$ is bijective. 
\end{conj}

In \cite{J}, Jelonek proved that if his conjecture is true in dimension $2n$ then so is the famous 

\begin{conj}[Jacobian conjecture \cite{K}]\label{JaCo}
	Let $f\colon\C^n\to\C^n$ be a polynomial map with nowhere zero Jacobian determinant. 
	Then $f$ is an automorphism (bijective with $f^{-1}$ a polynomial map).
\end{conj}

Jelonek \cite{J} also proved that Conjecture \ref{JeCo} is true under the additional assumption $\codim(S_f)\geq3$, as well as that it is true for $n=2$. 
We investigate this conjecture under the milder assumption that $f$ is semi-algebraic. 
Following Jelonek's proof, it is possible to conclude that semi-algebraic maps $f$ with nowhere zero Jacobian determinant and such that $\codim(S_f) \geq 3$ are globally injective (here the surjective may not be true as injective semi-algebraic maps need not be surjective). 
Moreover, Jelonek's conjecture remains true for semi-algebraic maps in dimension $n=2$, and we present a proof in Section \ref{sec:results}. 

We also remark that the known Pinchuk's counterexample \cite{Pi} to the so called \emph{real Jacobian conjecture}, i.e., a non injective polynomial map $p\colon\R^2\to\R^2$ with nowhere zero Jacobian determinant, satisfies $\codim \left(S_p\right) =  1$. 
Hence the assumption $\codim \left(S_f\right) \geq 2$ is necessary and it remains to investigate Conjecture \ref{JeCo} in dimension $n\geq3$. 

In \cite{FMS}, Fernandes, Maquera and the third named author of the present paper obtained the following topological result related to Conjecture \ref{JeCo}. 
Before the statement of the result we need some notation. 

For a given continuous submersion $g: \R^n \to \R^m$, $n>m$, we denote by $\F_g$ the foliation of $\R^n$ of codimension $m$ induced by the level sets of $g$. 
In the special case $m=n-2$, for each $k = 1, 2, \ldots, n-2$, we denote by $G_k\colon \R^n\to\R^{k}$ the submersion defined by $G_k=\left(g_1, g_2,\ldots,g_k\right)$. 
When the leaves of the two-dimensional foliation $\F_g = \F_{G_{n-2}}$ are simply connected or, equivalently, homeomorphic to $\R^2$, we say that \emph{$\F_{g}$ is a foliation by planes}. 
Finally, for a given positive integer $n$, we denote by $C_{n-2}$ the collection of subsets $\cset$ of $n-2$ elements of $\{1,\ldots,n\}$ such that $i_1<i_2<\cdots<i_{n-2}$. 

\begin{teo}[\cite{FMS}]\label{teo:FMS} 
Let $f = (f_1,\ldots,f_n)\colon\R^n\to\R^n$ be a semi-algebraic local diffeomorphism such that $\codim(S_f)\geq 2$. 
If
\begin{equation}\label{eq}
\F_{\ckif{n-2}} \textnormal{ is a foliation by planes for all }  
\cset\in C_{n-2},
\end{equation}
then $f$ is bijective.
\end{teo} 

As a consequence it was established the following semi-algebraic version of a polynomial result given previously by Gutierrez and Maquera in \cite{GM}, where $\spec(f)=\{\lambda \in \C\mid \lambda \mbox{ is an eigenvalue of }  Df(x),\ x\in \R^n\}$: 

\begin{cor}[\cite{FMS}]\label{cor:FMS}
Let $f\colon\R^3\to\R^3$ be a $C^2$ semi-algebraic map such that $\codim\left(S_f\right)\geq2$. 
If $\spec(f)\cap[0,\eps)=\vazio$ for some $\eps>0$, then $f$ is injective.
\end{cor} 

We remark that the above corollary follows essentially from Theorem \ref{teo:FMS} and an argument of approximating $f$ by $C^2$ semi-algebraic maps $g$ satisfying the analytical condition $\spec(g)\cap(-\eps,\eps)=\vazio$, as the later one implies the topological condition \eqref{eq} (see \cite[Theorem 1.1]{GM} and \cite[Corollary 1.5]{FMS} for details). 
The regularity $C^2$ is important in this corollary because the later implication needs this in \cite{GM}. 
Further, as observed in \cite[page 227]{FMS}, we cannot assure in Corollary \ref{cor:FMS} that $f$ is bijective (see also Example \ref{ex:SpecnoPS} below).  

Given a local diffeomorphism $f\colon\R^n\to\R^n$, it is a complicated task to check condition \eqref{eq}, so it is interesting to obtain results analogously to Corollary \ref{cor:FMS}, where analytical conditions provide condition \eqref{eq} and hence the global injectivity of $f$. 

In Section \ref{sec:results} we recall that each of the following conditions on a $C^2$ submersion $g: \R^3 \to \R$ guarantee that the foliation $\F_g$ is by planes. 
See Section \ref{sec:results} for the precise definition of the terms appearing below. 
\begin{enumerate}
\item[(a)]\label{cor:fib3} $g$ is a locally trivial fibration at the image. 
\item[(b)] $K_{\infty}(g)=\vazio$. 
\item[(c)] $g$ satisfies PS condition.
\item[(d)] $g$ satisfies 
$$
\int_{0}^{\infty}\inf_{\|x\|=r}\|\nabla g(x)\|\, dr=\infty. 
$$
\end{enumerate}	

So, by Theorem \ref{teo:FMS}, it readily follows 

\begin{cor}\label{pro:main1}
Let $f = (f_1,f_2,f_3)\colon\R^3\to\R^3$ be a $C^2$ semi-algebraic local diffeomorphism such that $\codim(S_f)\geq2$. 
If for each $i$, $f_i$ satisfies one of the conditions \emph{(a) -- (d)} above, then $f$ is bijective. 
\end{cor}

And in higher dimensions, also in Section \ref{sec:results} we recall that each of the following conditions on a $C^2$ submersion $g: \R^n \to \R^{n-2}$ guarantee that the foliation $\F_g$ is by planes: 
\begin{enumerate}
\item[(a$_h$)]\label{cor:fib} $g$ is a locally trivial fibration.  
\item[(b$_h$)] For all $(c_{1},\ldots,c_{k})\in\Imm G_{k}$, $k= 1,\ldots,n-3$, the functions $g_1$, $g_{2}|_{g_1^{-1}(c_{1})}$, $g_{3}|_{G^{-1}_{2}(c_{1},c_{2})}$,\ldots, $g_{n-2}|_{G^{-1}_{{n-3}}(c_{1},\ldots,c_{n-2})}$ are locally trivial fibrations at the image. 
\item[(c$_h$)] $K_\infty(g)=\vazio$. 
\item[(d$_h$)] For all $(c_{1},\ldots,c_{k})\in\Imm G_{{k}}$ and $k= 1,\ldots,n-3$, the sets $K_\infty(g_1)\cap\Imm(g_1)$, $K_\infty(g_{2}|_{g^{-1}_{1}(c_{1})})\cap\Imm(g_{2}|_{g^{-1}_{1}(c_{1})})$, $K_\infty(g_{3}|_{G^{-1}_{2}(c_{1},c_{2})})\cap\Imm(g_{3}|_{G^{-1}_{2}(c_{1},c_{2})}),\ldots,$ $K_\infty(g_{n-2}|_{G^{-1}_{{n-3}}(c_{1},\ldots,c_{n-3})})\cap\Imm(g_{n-2}|_{G^{-1}_{{n-3}}(c_{1},\ldots,c_{n-3})})$ are empty. 
\item[(e$_h$)]\label{cor:sab2} Given a $C^2$ submanifold $X$ of $\R^n$ and $x \in X$, we denote by $\nabla h_j|_X(x)$ the projection of $\nabla g_j(x)$ to the tangent space $T_xX$. 
Assume that 
\begin{align*}
\int_{0}^{\infty} \inf_{\|x\|=r}\|\nabla g_{1}(x)\|\, dr & =\int_{0}^{\infty}\inf_{\|x\|=r}\|\nabla h_{2}|_{g^{-1}_{1}(c_{1})}(x)\|\, dr \\
& = \int_{0}^{\infty}\inf_{\|x\|=r}\|\nabla h_{3}|_{G^{-1}_{2}(c_{1},c_{2})}(x)\|\, dr 
=\cdots \\
& = \int_{0}^{\infty}\inf_{\|x\|=r}\|\nabla h_{n-2}|_{G^{-1}_{n-3}(c_{1},\ldots,c_{n-3})}(x)\|\, dr \\
& =\infty
\end{align*}
for all $(c_{\cum},\ldots,c_{\cumk})\in\Imm G_k$ and $k= 1,\ldots,n-3$. 
\end{enumerate}

So by Theorem \ref{teo:FMS}: 
\begin{cor}\label{pro:main2}
Let $f = (f_1,\ldots,f_n)\colon\R^n\to\R^n$, $n > 3$, be a $C^2$ semi-algebraic local diffeomorphism such that $\codim(S_f)\geq2$. 
If for every $\cset\in C_{n-2}$ the map $g=F_{i_{n-2}}$ satisfies one of the conditions \emph{(a$_h$) -- (e$_h$)}, then $f$ is bijective. 
\end{cor}

The fact that each condition above implies that the related foliation is by planes follows from very well known facts, so our purpose in this part is merely expository and it is motivated by the fact that the above injectivity results are not explicitly presented in the literature, for $n\geq 3$. 
A further motivation is that when $n=2$, it is well known that each of the conditions (a) -- (d) on one of the components of $f$ guarantees the injectivity of $C^2$ local diffeomorphisms $f:\R^2\to\R^2$ without any other assumption (see \cite{GJLT,deM,R2,S}), as it is the case of the spectral condition in Corollary \ref{cor:FMS} \cite{CGL, FGR, G}. 
On the other hand corollaries \ref{pro:main1} and \ref{pro:main2} are not true without the condition $\codim(S_f) \geq 2$, as shown by Example \ref{cex:dim3}. 
See Section \ref{examples} for more information. 

The relations between conditions (a), (b) and (c) are known and will be recalled in Section \ref{sec:results}. 
But up to our knowledge, there are no comparisons of these conditions with the spectral condition of Corollary \ref{cor:FMS}. 
In Section \ref{examples}, invoking results from Section \ref{sec:results}, we present examples proving that all the conditions (a) -- (d) are different of the spectral condition of Corollary \ref{cor:FMS}. 

Given a complex map $f\colon\C^m\to\C^m$ we can treat it as a real map $f_{\R} = (f_1,\ldots,f_n)\colon\R^{n}\to\R^{ n}$, where $n = 2m$. 
As a consequence of our reasonings in Section \ref{sec:results}, we obtain the following  equivalence with the Jacobian Conjecture \ref{JaCo} by means of fibrations, this is related to an equivalence in $\C^2$ given by L\^e and Weber in \cite{LW}, see also \cite[Conjecture 1.7]{BDVS}: 

\begin{pro}\label{pro:equiv} 
	Let $f\colon\C^m\to\C^m$, $m\geq2$, be a polynomial map with nowhere zero Jacobian determinant. 
	Then $f$ is an automorphism if and only if for its real representation $f_{\R}=(f_1,\ldots,f_n)\colon\R^{n}\to\R^{n}$, the maps $F_{\cki{n-2}}$ are locally trivial fibrations for all  combinations $\cset\in C_{n-2}$. 
\end{pro} 

As a final application of Section \ref{sec:results}, in Section \ref{sec:rotation} we discuss about a key difference between spectral hypotheses and the conditions discussed in this work. 
Given a local diffeomorphism  $f \colon\R^n\to\R^n$ and a linear isomorphism $A\colon\R^n\to\R^n$, we clearly have $\spec(f) = \spec(A\circ f \circ A^{-1})$. 
That is, spectral conditions are invariant by conjugation with linear isomorphisms. 
This property was important in the proofs of global injectivity results related to spectral conditions in the works \cite{CGL,FGR,G,GM}. 
In Example \ref{cex:1}, we show that the same is false for regularity conditions at infinity, the integral condition (d), locally trivial fibrations and even for the topological condition \eqref{eq}. 

Such discussion is pertinent to clarify some misunderstandings such as in \cite[Lemma 8]{MV-S}, where it is claimed that condition \eqref{eq} is invariant by composing $f$ through the left by linear isomorphisms in contrast to Example \ref{cex:1}. 
Anyway, in Section \ref{sec:rotation} we give an alternative proof to \cite[Theorem 1]{MV-S} without using \cite[Lemma 8]{MV-S}. 
Another misunderstanding appears in the proof of \cite[Theorem 2.3]{B}, where it is claimed that a generalization of (d) to higher dimensions is  invariant by composing $f$ through the left by rotations, in contrast to Example \ref{ex:1}. 
Further, Example \ref{ex:2} shows that \cite[Theorem 2.3]{B} needs to be slightly reformulated. 
We finish proposing a modified version of \cite[Theorem 2.3]{B}, namely, Theorem \ref{novo} below, arguing in the sequel that this new version is sufficient to obtain the global injectivity results of \cite{B}. 

\section{Fibrations, regularity conditions and global inversion results}
\label{sec:results}

\subsection{Fibrations}

Before introducing the regularity conditions that ensure condition \eqref{eq}, we will present some intermediate conditions based on the concept of locally trivial fibrations. 

\begin{defi}\label{def:bv} 
We say that a continuous map $g: X \to Y$ between topological spaces $X$ and $Y$ is a \emph{trivial fibration} if there exists a topological space $\mathcal{F}$ and a homeomorphism $\varphi\colon \mathcal{F}\times Y\to X$ such that $pr_2 = g\circ\varphi$ is the projection on $Y$.  Note that if $g$ is a trivial fibration, then each fiber $g^{-1}(y)$, for $y\in Y$, is homeomorphic to $\mathcal{F}$.
We say further that $g$ is a \emph{locally trivial fibration at $y\in Y$} if there exists an open neighborhood $U$ of $y$ in $Y$ such that $g|_{g^{-1}(U)}\colon g^{-1}(U)\to U$ is a trivial fibration. 
We denote by $B(g)$ the set of points of $Y$ where $g$ is not a locally trivial fibration. 
The set $B(g)$ is usually called the \emph{bifurcation} (or \emph{atypical}) set of $g$. 
In case $B(g)$ is the empty set we simply say that $g$ is a \emph{locally trivial fibration}. 
\end{defi}

Let us recall a known property on trivial fibrations, see for instance \cite[11.6]{St51}. 

\begin{pro}\label{pro:eq}
	If $g\colon X\to Y$ is a locally trivial fibration and $Y$ is a contractible space, then $g$ is a trivial fibration.
\end{pro}

This property will be applied to relate locally trivial fibrations to condition \eqref{eq} in the following 

\begin{pro}\label{pro:fibration}
Let $g:\R^n\to\R^{n-2}$, $n\geq3$, be a continuous map such that $B(g) = \vazio$. 
Then the level sets of $g$ are simply connected. 
In particular, $\F_g$ is a foliation by planes.
\end{pro}

\begin{proof}
Let $c \in\R^{n-2}$. 
Since $\R^{n-2} $ is contractible, it follows from Proposition \ref{pro:eq} that $g^{-1}(c)\times\R^{n-2}$ is homeomorphic to $\R^n$. 
Thus $g^{-1}(c)$ is simply connected. 
In particular, the leaves of $\F_g$ are simply connected, and hence this foliation is by planes.
\end{proof}

This basic property is enough to obtain the previously enunciated equivalence with the Jacobian Conjecture \ref{JaCo} by means of fibrations:

\begin{proof}[Proof of Proposition \ref{pro:equiv}]
If $f$ is an automorphism then $f_{\R} = \left(f_1,\ldots,f_n\right)\colon\R^{n}\to\R^{n}$ is a global diffeomorphism and so the map $F_{\cki{n-2}}$ is a locally trivial fibration, for all combinations $\cset\in C_{n-2}$. 
Indeed, consider without loss of generality the map $F_{n-2}=(f_1,\ldots,f_{n-2})$ and $c\in\R^{n-2}$. 
Denoting $t=(t_1,\ldots,t_{n-2})$, it follows that $\varphi\colon F_{n-2}^{-1}(c)\times\R^{n-2} \to \R^n$ given by $\varphi(x,t)=f_{\R}^{-1}(t,f_{n-1}(x),f_n(x))$, with inverse $\varphi^{-1}(y) = \left(f_{\R}^{-1}(c, f_{n-1}(y), f_n(y)), F_{n-2}(y) \right)$, is a homeomorphism such that $F_{n-2}\circ \varphi=pr_2$. 
	
On the other hand, for a polynomial map $f\colon\C^m\to\C^m$ with nowhere zero Jacobian determinant, it follows from \cite{J99} and \cite{J93} that the set $S_f$ of non-proper points of $f$ is empty or has (complex) codimension $1$. 
In any case the real representation of $f$, $f_{\R}=(f_1,\ldots,f_n)\colon\R^{n}\to\R^{n}$, has $S_{f_{\R}}$ with (real) codimension greater than or equal to two. 
If $F_{\cki{n-2}}$ is a locally trivial fibration for all combinations $\cset\in C_{n-2}$, it follows from Proposition \ref{pro:fibration} that $f_{\R}$ satisfies condition \eqref{eq}, and hence $f_{\R}$ is bijective from Theorem \ref{teo:FMS}. 
Thus $f$ is bijective. 
The inverse map is polynomial from \cite{CR}. 
\end{proof}

We notice that by definition locally trivial fibrations are surjective maps. 
But for instance, although $f\colon\R \to\R$ defined by $f(x) = x + \sqrt{1+x^2}$ is not surjective, the restriction of this function to its image $f\colon\R \to(0, \infty)$ is a locally trivial fibration. 
We include this cases with the following definition, where for a map $f: X \to Y$ we denote $\Imm(f) = f(X)$. 

\begin{defi}\label{def:fibi}
We say that a continuous map $g\colon X \to Y$ is a \emph{locally trivial fibration at the image} if the restriction $g\colon X\to \Imm(g)$ is a locally trivial fibration. 
Equivalently, if $BI(g):=B(g)\cap \Imm(g)=\vazio$.
\end{defi}

\begin{pro}\label{pro:fibration2}
Let $g:\R^n\to\R^{n-2}$, $n\geq3$, be a continuous map. 
Assume that we have 
$$
BI(g_{1}) = BI\left(g_{{k+1}}|_{G^{-1}_{{k}}(c_{1},\ldots,c_{k})}\right)= \vazio, 
$$
for all $(c_{1},\ldots,c_{k})\in\Imm G_{{k}}$, $k = 1,\ldots, n-3$. 
Then for each $c\in\Imm g$ the level sets $g^{-1}(c)$ are simply connected. 
In particular, $\F_g$ is by planes.
\end{pro}

\begin{proof}
Given $c_{1}\in\Imm (g_{1})$, since $BI(g_1)=\vazio$ and $\Imm (g_1)$ is an open interval of $\R$ and so it is contractible, it follows from Proposition \ref{pro:eq} that $g_{1}^{-1}(c_{1})\times\Imm( g_{1})$ is homeomorphic to $\R^n$. 
In particular, the level set $g_1^{-1}(c_{1})$ is simply connected. 

Now since $BI(g_{2}|_{g_1^{-1}(c_{1})})=\vazio$ and $\Imm (g_{2}|_{g_{1}^{-1}(c_{1})})$ is contractible, it follows again from Proposition \ref{pro:eq} that $(g_{2}|_{g_{1}^{-1}(c_{1})})^{-1}(c_{2})\times\Imm (g_{2}|_{g_{1}^{-1}(c_{1})}) $ is homeomorphic to $g_{1}^{-1}(c_{1})$ for all $c_{2}\in \Imm (g_{2}|_{g_{1}^{-1}(c_{1})})$. 
In particular, analogously to the above paragraph, we conclude that $G_{2}^{-1}(c_{1},c_{2}) = (g_{2}|_{g_{1}^{-1}(c_{1})})^{-1}(c_{2})$ is simply connected, for any $(c_{1},c_{2})\in \Imm(G_{2})$. 

From the simply connectedness of $G_{2}^{-1}(c_{1},c_{2})$ and hypotheses we similarly conclude the simply connectedness of $G_{3}^{-1}(c_{1},c_{2},c_{3}) = (g_{3}|_{G_{2}^{-1}(c_{1},c_{2})})^{-1}(c_{3})$. 

By repeating this process we inductively reach to the end of the proof. 
 
In particular the level sets $g^{-1}(c)$ are simply connected for all $c\in\Imm g$ and so $\F_g$ is by planes. 
\end{proof}

From Proposition \ref{pro:fibration} (respectively Proposition \ref{pro:fibration2}) it is clear that assuming (a$_h$) (respectively (a) or (b$_h$)) it follows that the foliation $\F_g$ is by planes. 

\subsection{Regularity conditions}

The bifurcation locus $B(g)$ of a map $g\colon \K^n \to \K^m$, $\K=\R$ or $\C$, $n\geq m$, is fully characterized only for real or complex polynomial maps in special cases (see for instance \cite{HL, JTC, Pa} for $\K=\C$ and \cite{CP, JTR, TZ} for $\K=\R$). 
In general one may control $B(g)$ by using certain criteria of regularity at infinity or asymptotic regularity.

Here, we begin recalling Palais-Smale condition. In the following, the notation $u_l\to\infty$ means that the sequence $\{u_l\}$ has no convergent subsequences. 
Also $\Sing(g)$ stands for the set of critical points of $g$. 

\begin{defi} \label{def:PS} 
Let $X$ be a $C^1$ Riemannian manifold and $g: X\to\R$ be a $C^1$ function. 
We say that $g$ satisfies the \emph{Palais-Smale condition} (in short \emph{PS condition}) if any sequence $\{u_l\}_{l\in\N}$ in $X$ such that $\{g(u_l)\}$ is bounded and $\nabla g(u_l)\to 0$ has a convergent subsequence. 
In the special case of $\Sing(g)=\vazio$, this is equivalent to: for each sequence $\{u_l\}_{l\in\N}$ in $X$ with $u_l\to\infty$ and $g(u_l) \to c$, $c \in \R$, there exists $\varepsilon > 0$ such that $||\nabla g(u_l)|| \geq \varepsilon$, where $||\cdot||$ is a Riemannian metric of $X$.
\end{defi}

Let us now recall a regularity condition introduced by Rabier in \cite{R2}, which is also known in the literature by Malgrange condition when $m=1$.

\begin{defi} For a linear map $A\colon\R^n\to\R^m$ we set
$$\nu(A):=\inf_{\|\psi\|=1}\|A^*(\psi)\|,$$
where $A^*$ is the adjoint map of $A$ and $\psi\in(\R^m)^*$.
\end{defi}

\begin{defi}\label{def:Malg}
Let $g\colon X\to\R^m$ be a $C^2$ map, where $X$ is a $C^2$ submanifold of $\R^n$ of dimension $k\geq m$. 
We define the set of \emph{asymptotic critical values of $g$ at infinity} by
\begin{align*}
K_{\infty}(g) := & \Big\{y\in\R^m\ | \ \exists \{ u_{l}\}_{l\in \N} \subset X, \,\, u_l\to\infty, \\
  & \phantom{\big\{ } \lim_{l\to\infty}g(u_l)=y \mathrm{\ and\ }\lim_{l\to\infty}\|u_l\|\nu(Dg (u_l)|_{T_{u_l}X})=0\Big\}.
\end{align*}
We say that the map $g$ satisfies  \emph{Rabier condition} if $K_\infty(g)=\vazio$.
\end{defi}

\begin{rem} 
There is a wide literature treating different regularity conditions and their comparisons, see for instance \cite{DRT,Ga,J2,KOS,ST2} as well as the references therein. 
\end{rem}

For $m=1$, it is clear that if $g$ satisfies PS condition then $K_\infty(g)=\vazio$. 
But the reciprocal fails, as shown by the counterexample given by P\u{a}unescu and Zaharia \cite{PZ}, $g(x,y,z) = x - 3 x^3 y^2 + 2 x^4 y^3 + y z$. 
The sequence $(l, 1/l, 0)$ proves that it does not satisfy PS condition. 
On the other hand, we have $K_\infty(g) = \vazio$,  because if $u_l = (x_l, y_l,z_l) \to \infty$ and $\|u_l\| \|\nabla g(u_l)\| \to 0$, it follows that $\nabla g(u_l) \to 0$ and so $\partial g(u_l)/\partial x = 1 - 9 (x_l y_l)^2 + 8 (x_l y_l)^3\to 0$, proving that $x_l y_l \nrightarrow 0$, and hence $\|u_l\| |\partial g(u_l)/\partial_z| \nrightarrow 0$, a contradiction with $\|u_l\| \|\nabla g(u_l)\| \to 0$.

In \cite{R2} Rabier proved the following result in a broader setting, namely for $C^2$ maps $g\colon M\to N$, where $M$ and $N$ are Finsler manifolds. 

\begin{teo}[{\cite[Rabier]{R2}}]\label{Rabier}
Let $g\colon X\to\R^m$ be a $C^2$ map, where  $X$ is a $C^2$ submanifold of $\R^n$ of dimension $k\geq m$. 
Then $$
B(g)\subset K_{\infty}(g)\cup g\left(\Sing (g)\right). 
$$
\end{teo}

This inclusion can be strict even for polynomial submersions, as can be viewed by the example of Henry King \cite[Example 3.4]{TZ}, that we will consider in detail in Example \ref{King} below. 

From this theorem and Proposition \ref{pro:fibration} (respectively Proposition \ref{pro:fibration2}), it follows that assuming (b), (c) or (c$_h$) (respectively (d$_h$)) then the foliation $\F_g$ is by planes.

The following lemmas are certainly well known. 
For a proof of Lemma \ref{lem:1} see, for instance, \cite[Lemma 2.2]{NX}. 
Lemma \ref{lem:2} is proved in \cite{R2} and in a more general setting in \cite[Theorem 1]{deM}.

\begin{lem}\label{lem:1} 
Let $Z\colon\R^n\to\R^n$ be a non-vanishing smooth vector field such that
$$
\int_{0}^{\infty}\inf_{\|x\|=r}\|Z(x)\|\, dr=\infty.
$$
Then the vector field ${Z(x)}/{\|Z(x)\|^2}$ is complete.
\end{lem}

\begin{lem}\label{lem:2} 
Let $X$ be $C^2$ manifold. 
If $h\colon X\to\R$ is a $C^2$ function, $\nabla h(x)$ is non-vanishing for all $x\in X$ and ${\nabla h(x)}/{\|\nabla h(x)\|^2}$ is a complete vector field, then $B(h)=\vazio$. 
\end{lem}

From lemmas \ref{lem:1} and \ref{lem:2} together with Proposition \ref{pro:fibration} (respectively Proposition \ref{pro:fibration2}), it follows that (d) (respectively (e$_h$)) implies that $\F_g$ is by planes. 

\subsection{Bidimensional semi-algebraic Jelonek's conjecture}
We finish this section with a proof of a semi-algebraic version of Conjecture \ref{JeCo} in dimension $2$, as we did not find one in the literature.  
\begin{pro}
Let $f\colon\R^2\to\R^2$ be a $C^1$ semi-algebraic map with nowhere zero Jacobian determinant. 
If $\codim(S_f)\geq 2$, then $f$ is bijective. 
\end{pro}
\begin{proof} 
The objective of the proof is to conclude that $S_f$ is the empty set, so $f$ will be bijective by the Hadamard classical invertibility result. 

From the assumptions it follows that $\dim(S_f) = 0$, and so $S_f$ is a finite set. 
Let $R_1>0$ big enough such that the ball centered in $0$ with radius $R_1$, $B_{R_1}$, has $S_f$ as a subset. 
For any $R>R_1$, the circle centered in $0$ with radius $R$, $S_R$, is such that $f^{-1}(S_R)$ is a manifold of dimension $1$, and so, as $S_R$ does not intersect $S_f$, it follows that the connected components of $f^{-1}(S_R)$, say $C_1, \ldots, C_k$, are all homeomorphic to a circle. 

For each $j=1,\ldots, k$, it follows by the Jordan curve theorem that $\R^2\setminus C_j$ has two connected components, one bounded and another unbounded. 
We denote by $L_j$ the bounded one and by $U_j$ the unbounded one. 

We claim that $f(L_j) = B_R$. 
Indeed, $f(L_j) \subset B_R$ because otherwise there would exist $x_0 \in L_j$ such that  $\|f(x_0)\|  \geq R$. 
Then, as $L_j \cup C_j$ is a compact set and $\|f(x)\| = R$ in $C_j$, it follows that the function $L_j  \cup C_j \ni x \to \|f(x)\|^2 \in \R$ has a local maximum with value $\geq R$ in the open set $L_j$, a contradiction with the fact that the Jacobian matrix of $f$ is invertible in each point. 
Further, as $L_j$ is an open set and $f$ is an open mapping, we have that $f(L_J)$ is an open set, and hence $f(L_j)$ is an open set of $B_R$. 
Now we will prove that  $f(L_j)$ is a closed set of $B_R$, then, as $B_R$ is a connected set, it follows that $f(L_j) = B_R$ as claimed. 
Let $\{f(x_n)\}$ be a sequence of $f(L_j)$ such that $\lim f(x_n) = y \in B_R$. 
As $L_j \cup C_j$ is a compact set, we may consider a subsequence $\{x_{n_k}\}$ of $\{x_{n}\}$ such that $\lim x_{n_k} = x \in L_j \cup C_j$. 
Since $f$ is continuous we have $\lim f(x_{n_k}) = f(x)$ and so $y = f(x)$. 
Since we have taken $y \in B_R$ and $f(C_j) \subset S_R$, it follows that $x \in L_j$ and so $y\in f(L_j)$. 

Let $X=\overline{\cup_{j=1}^k L_j}$. 
Clearly $f(\R^2 \setminus X) \not \subset B_R$, otherwise $f(\R^2) \subset\overline{B_R}$ and so $S_f$ would not be a finite set. 
Since the set $\R^2 \setminus X$ is connected from Alexander's duality, it follows that $f^{-1} (\overline{B_R}) = X$. 
But as $X$ is compact and $S_f \subset B_R$, it follows that $S_f$ must be the empty set. 
\end{proof}

\section{Examples}\label{examples}

Let $f = (f_1,f_2)\colon\R^2\to\R^2$ be a $C^1$ local diffeomorphism. 
It is well known that assuming (a), (b), (c) or (d) on $f_1$ (without any other hypothesis) it follows that $f$ is injective, see for instance  \cite{GJLT, deM, S,R2}. 

On the other hand, for $n>2$, only conditions (a) -- (d) are not sufficient to the injectivity. 
More precisely, our first example exhibits a non injective polynomial local diffeomorphism of $\R^3$ with each coordinate functions satisfying (a), (b), (c) and (d). 
So the condition about the set of non-proper points in corollaries \ref{pro:main1} and \ref{pro:main2} can not be  removed for $n>2$. 

\begin{ex}\label{cex:dim3}
Consider the already mentioned non-injective polynomial map $p=(p_1,p_2)\colon\R^2\to\R^2$ with nowhere zero Jacobian determinant, given by Pinchuk in \cite{Pi}. 
We have $S_p\neq\vazio$. 
So by \cite[Theorem 6.4]{J} it follows that $\codim (S_p) = 1$. 
Let $f = (f_1,f_2,f_3)\colon\R^3\to\R^3$ be the polynomial map defined by 
$$
f(x_1,x_2,x_3) = (x_3 + p_1(x_1,x_2), x_3 + p_2(x_1,x_2), x_3). 
$$
Clearly $f$ is a non-injective polynomial map with nowhere zero Jacobian determinant. 
Moreover, since $\|\nabla f_i\| \geq 1$, it follows that $f_i$ satisfies the integral condition (d), the PS condition as well as $K_{\infty}(f_i)=\vazio$, and hence from Theorem \ref{Rabier}, $f_i$ is a locally trivial fibration, for $i = 1,2,3$. 
\end{ex}

Now we compare some of the conditions ensuring global injectivity in this paper. 
We begin remarking that, under the assumption $\codim (S_f) \geq 2$ on a $C^2$ semi-algebraic local diffeomorphism $f: \R^3 \to \R^3$, it follows from the proof of Proposition \ref{pro:equiv} that the topological condition \eqref{eq} is equivalent to condition (a) applied to each coordinate of $f$. 
On the other hand, the following example shows that conditions (b), (c) and (d) applied to the coordinates of $f$ are stronger than \eqref{eq} (and consequently stronger than (a)). 

\begin{ex}\label{King}
Consider the Henry King's example $h(x_1, x_2) = x_2 (2x_1^2x_2^2 - 9x_1 x_2 + 12)$ and define, analogously to \cite[Example 3.4]{TZ}, the map 
$$
\widetilde{f}(x_1,x_2) = \left(h(x_1, x_2), \frac{x_1}{2x_1^2x_2^2 - 9x_1 x_2 + 12} \right). 
$$
It is simple to see that $\widetilde{f}\circ \widetilde{f}(x_1,x_2) = (x_1,x_2)$, and so $\widetilde{f}$ is a global diffeomorphism of $\R^2$. 
The map $f(x_1,x_2,x_3) = \left(\widetilde{f}(x_1,x_2), x_3\right)$ is then a global diffeomorphism of $\R^3$ (thus $S_f = \vazio$). 
Hence $f_j$, $j=1,2,3$, are locally trivial fibrations and so condition \eqref{eq} is also valid. 

On the other hand, $0\in K_\infty(h)$ (and so $h$ does not satisfy PS condition), because the sequence $u_l = (l, 1/l)$ is such that $h(u_l)\to 0$ and $\|u_l\|\|\nabla h(u_l)\| \rightarrow 0$. 
Further, $\int_{0}^{\infty}\inf_{\|x\|=r}\|\nabla h(x_1,x_2)\|dr < \infty$, because $\inf_{\|x\|=r} \|\nabla h(x_1, x_2)\| \leq \|\nabla h(t,1/t)\| = 5/t^2 \leq 10/r^2$ (for $t^2 = \left(r^2+\sqrt{r^4-4}\right)/2$ which satisfies $t^2+1/t^2 = r^2$). 
Then $f_1(x_1,x_2,x_3) = \widetilde{f}_1(x_1,x_2) = h(x_1, x_2)$ satisfies neither $K_\infty(f_1)=\vazio$ (and consequently PS condition) nor condition (d). 
\end{ex}

With the following examples we conclude that, even for semi-algebraic local diffeomorphisms $f\colon\R^3\to\R^3$ with $\codim S_f\geq2$, the spectral condition of Corollary \ref{cor:FMS} can have no relation with conditions (a), (b), (c) and (d) applied to the coordinates of $f$. 

\begin{ex}\label{ex:PSnoSpec} 
Let $f=(f_1,f_2,f_3)\colon\R^3\to\R^3$ be the polynomial map 
$f(x_1,x_2,x_3) = \left(x_2-x_1^2,x_1,x_3\right)$. 
This is a global diffeomorphism whose inverse is $f^{-1}(x_1,x_2,x_3) = \left(x_2, x_1 + x_2^2, x_3 \right)$. 
Thus $S_f = \vazio$. 
Since $\left|\nabla f_i\right|\geq1$, we have that $f_i$ satisfies the integral condition (d) as well as the PS condition, and consequently $K_\infty(f_i)=\vazio$ and $f_i$ is locally trivial fibration, for all $i = 1, 2, 3$. 
On the other hand the eigenvalues of the Jacobian matrix $Df(x_1,x_2,x_3)$ are given by $\lambda_1=1$ and $\lambda_{\pm}(x_1,x_2,x_3)=-x_1\pm\sqrt{x_1^2+1}$, hence  $\spec(f) = \R\setminus\{0\}$, i.e., $f$ does not satisfy the spectral condition of Corollary \ref{cor:FMS}. 
\end{ex}

The following example appeared in \cite{FMS}. 

\begin{ex}\label{ex:SpecnoPS}
Let $f=(f_1,f_2,f_3)\colon\R^3\to\R^3$ be defined by
$$
f(x_1,x_2,x_3) = \left(\sqrt{1+x_1^2}-(x_2^2+x_3^2+1)x_1,x_2,x_3 \right).
$$ 
This is a smooth semi-algebraic map such that $\spec(f) = (-\infty,0) \cup \{1\}$, and so $f$ satisfies the spectral condition of Corollary \ref{cor:FMS}. 
Hence $f$ is globally injective. 
Here $\Imm f = \R^3\setminus E$ where $E=\{(x_1,0,0)\ | \ x_1\leq0\}$ and so $S_f = E$ has dimension $1$. 
But since $\F_{f_1}$ is not by planes (for instance $f_1^{-1}(0)$ is connected but not simply connected) it follows that $f_1$ does not satisfy any of the conditions (a) -- (d). 
\end{ex}

\section{Compositions with linear isomorphisms}\label{sec:rotation} 

Let $A: \R^n \to \R^n$ be a $C^1$ diffeomorphism. 
First we compose through the right. 

Let $f = (f_1, \ldots, f_n): \R^n \to \R^n$ be a local homeomorphism. 
Since for each $k = 1,\ldots, n-1$ the foliations $\F_{F_{i_k}}$ and $\F_{F_{i_k}\circ A}$ are conjugated, because $\{x\in\R^n\ |\ F_{\cki{k}}\circ A (x)=c\} = A^{-1}\left(\{x\in\R^n\ |\ F_{\cki{k}}(x)=c\}\right)$ for all $c\in\R^k$, it follows in particular that $\F_{F_{i_{n-2}}}$ is by planes if and only if $\F_{F_{i_{n-2}}\circ A}$ is by planes. 

Also given $c \in \R^k$, $U$ a neighborhood of $c$ and $\phi \colon F_{\cki{k}}^{-1}(c)\times U\to F_{\cki{k}}^{-1}(U)$ a homeomorphism such that $F_{\cki{k}}\circ \phi = pr_2$, by taking $\varphi = \overline{A}^{-1} \circ \phi \circ (\overline{A} \times id)$, where $\overline{A}$ is the restriction of $A$ to the set $A^{-1}(F_{\cki{k}}^{-1}(U))$ and $id$ is the identity in $U$, it follows that $\varphi$ is a homeomorphism from $(F_{\cki{k}} \circ A)^{-1}(c) \times U$ onto $(F_{\cki{k}}\circ A)^{-1} (U)$ such that $(F_{\cki{k}} \circ A) \circ \varphi = pr_2$. 
This proves that $F_{\cki{k}}$ is a locally trivial fibration if and only if $F_{\cki{k}}\circ A$ is a locally trivial fibration. 
Similar reasonings verify this to locally trivial fibrations at the image. 

Now in the particular case of $A$ being a linear isomorphism of $\R^n$, it is also true that a component function $f_i$ of a $C^1$ map $f = (f_1, \ldots, f_n )$ satisfies PS condition if and only if $(f\circ A)_i = f_i \circ A$ does.
Indeed, this is a direct consequence of the chain rule $\nabla (f_i\circ A)(x) = \nabla f_i(A(x))\cdot A$. 
It follows analogously that $K_\infty(F_{\cki{k}})=\vazio$ if and only if $K_\infty(f\circ A)_{\cki{k}}=\vazio$. 
This may be false if $A$ is not linear. 
In fact, let $\widetilde{f}$ be the map defined in Example \ref{King}. 
As proved there, $\widetilde{f}_1$ satisfies neither $K_\infty(\widetilde{f}_1)=\vazio$ nor PS condition. 
But after the change of variables $A = \widetilde{f}$, we have $\widetilde{f}_1 \circ A(x_1,x_2) = x_1$, which obviously satisfies $K_\infty(\widetilde{f}_1 \circ A)=\vazio$ and PS conditions.

On the other hand, composing through the left is more involved, as the next example shows 

\begin{ex}\label{cex:1} 
Let $f = (f_1,f_2,f_3)\colon\R^3\to \R^3$ be defined by
\begin{equation*}\label{mapF}
f(x_1,x_2,x_3) = \left(x_1,x_2,\sqrt{2}s(x_3)-\sqrt{x_1^2+x_2^2+s(x_3)^2}\right), 
\end{equation*}
where $s\colon\R\to(0,\infty)$ is the $C^\infty$ semi-algebraic increasing diffeomorphism given by $s(t)=\sqrt{1+t^2}+t$. 
It follows that $f$ is a $C^{\infty}$ semi-algebraic map satisfying, for all $x = (x_1,x_2,x_3)\in\R^3$, 
$$
\det Df(x) = \frac{\partial f_3}{\partial x_3}(x)=\left(\sqrt2-\displaystyle\frac{1}{\sqrt{(x_1^2+x_2^2)s(x_3)^{-2}+1}}\right)s'(x_3)>0,
$$ 
and so $f$ is a local diffeomorphism. 
Moreover, $\partial f_3(x)/\partial x_3 > 0$ also implies that for each $(x_1,x_2)$ the map $\R\ni x_3\mapsto f_3(x_1,x_2,x_3)$ is injective, hence $f$ is injective and so a global diffeomorphism from $\R^3$ onto its image $f(\R^3)$. 

Since $\lim_{x_3\to-\infty}s(x_3)=0$ and $\lim_{x_3\to \infty}s(x_3)=\infty$, it follows that for each $(x_1,x_2)$ 
$$
\lim_{x_3\to-\infty}f_3(x_1,x_2,x_3)=-\sqrt{x_1^2+x_2^2},
$$ 
and
$$
 \lim_{x_3\to\infty}f_3(x_1,x_2,x_3)=\lim_{x_3\to \infty}\left(\sqrt{2}-\sqrt{\frac{x_1^2+x_2^2}{s(x_3)^{2}}+1}\ \right)s(x_3)=\infty,
$$
this shows that
\begin{equation*}\label{ceeq1}
f\left(\R^3\right)=\left\{(x_1,x_2,x_3)\in\R^3\ |\ x_3>-\sqrt{x_1^2+x_2^2}\right\}.
\end{equation*}
It is clear that the plane $\{x_3=c\}$ intersects $f(\R^3)$ in a connected non simply connected set if $c\leq0$ and in a simply connected set if $c>0$. 
Thus $f_3^{-1}(c)$ is a connected non-simply connected set for $c\leq0$ and it is a simply connected set for $c > 0$. 
That is, $f$ does not satisfy \eqref{eq}, because $\F_{f_3}$ is not by planes. Consequently, by Proposition \ref{pro:fibration}, $B(f_3)\neq\vazio$, by Theorem \ref{Rabier}, $K_\infty(f_3)\neq\vazio$ and by lemmas \ref{lem:1} and \ref{lem:2} the integral condition (d) is also not satisfied.

Now taking the linear isomorphism of $\R^3$ induced by the matrix
\begin{equation*}\label{hdt67}
A=\left(
\begin{array}{ccc}
1      & 0      & 0\\
0      & 1      & 0\\
1      & 0     & -\frac{1}{2}
\end{array}
\right),
\end{equation*}
and defining
$
g(x) \hspace{-.05cm} = \hspace{-.05cm} A\circ f(x_1,x_2,x_3) \hspace{-.05cm} = \hspace{-.05cm} \left(x_1, x_2, x_1 -\frac{\sqrt{2}}{2}s(x_3)+\frac{1}{2}\sqrt{x_1^2+x_2^2+s(x_3)^{2}}\right),
$
we observe that
$$
\partial_1 g_3(x_1,x_2,x_3)=1+\frac{1}{2}\frac{x_1}{\sqrt{x_1^2+x_2^2+s(x_3)^{2}}}>\frac{1}{2},
$$
and so $\|\nabla g_3(x_1,x_2,x_3)\|>1/2$, for each $(x_1,x_2,x_3)\in\R^3$. 
Hence $g_3$ satisfies PS condition and, in particular, $K_\infty(g_3) = \vazio$ and also satisfies the integral condition (d). 
It is clear that the same occurs with $g_1$ and $g_2$. 
In particular, by Theorem \ref{Rabier}, $B(g_i)=\vazio$ and, by Proposition \ref{pro:fibration}, $\F_{g_i}$ is by planes for all $i\in\{1,2,3\}$. 
\end{ex}

With the comments above about invariance by composing through the right and the Example \ref{cex:1} we may conclude that, contrary to what happens with spectral hypotheses, the PS condition, $K_\infty(\cdot)=\vazio$, the integral condition (d), the locally trivial fibrations and included the topological condition \eqref{eq} used in the global injectivity results presented in Theorem \ref{teo:FMS} and corollaries \ref{pro:main1} and \ref{pro:main2} are not invariant if we conjugate $f$ by a linear isomorphism. 

\cite[Lemma 8]{MV-S} claims that condition \eqref{eq} is invariant by composing through the left by linear isomorphisms, in contrast with Example \ref{cex:1}. 
Although this lemma was used to prove Theorem 1 of \cite{MV-S}, the result remains true: 
\begin{teo}[Theorem 1 in \cite{MV-S}]
\label{teo:injectivityn} Let $f=(f_1,\ldots,f_n):\R^n\to\R^n$ be a polynomial local diffeomorphism. 
Then the leaves of $\F_{i_1\ldots i_{n-2}}$ are simply connected, for all $(n-2)$-combinations $\{i_1,\ldots,i_{n-2}\}$ of $\{1,\ldots,n\}$, and $\codim(S_F)\geq 2$ if, and only if, $f$ is a bijection.
\end{teo}

\begin{proof}[An alternative proof to Theorem \ref{teo:injectivityn}]
It is clear that if $f$ is a bijection, then $f^{-1}$ is a diffeomorphism and so each leaf of $\F_{i_1\ldots i_{n-2}}$ is the image of a plane through $f^{-1}$, hence simply connected. 
Also $S_f = \vazio$. 

To prove the converse result in the original paper \cite{MV-S}, it was assumed, by contradiction, that $f$ was not injective, which implies that the set of non-proper points $S_f$ is non-empty and has dimension $n-2$. 
Then by letting a smooth point $p = (p_1,\ldots,p_n)$ of $S_f$, it was constructed a compact neighborhood $W$ of $p$ such that $f^{-1}(W)$ is compact, a contradiction. 
In such construction, the mentioned Lemma 8 was used to position $S_f$ in a suitable way such that the specific hyperplanes $\{x_1=p_1\}$, $\{x_2=p_2\}$, $\ldots$, $\{x_{n-2}=p_{n-2}\}$ are transversal to the tangent space $T_pS_f$ at the point $p$. 
Here we observe that this construction can be done without repositioning $S_f$, by simply choosing the hyperplanes  $\{x_{i_1}=p_{i_1}\}$, $\{x_{i_2}=p_{i_2}\}$, $\ldots$, $\{x_{i_{n-2}}=p_{i_{n-2}}\}$ that are transversal to $T_pS_f$, with $S_f$ on its original position. To do this it is enough to take a basis $\{v_1,\ldots,v_{n-2}\}$ of $T_pS_f$ and complete this to a basis $\{v_1,\ldots,v_{n-2},e_{i_{n-1}},e_{i_n}\}$ of $\R^n$, where $e_{i_{n-1}}$ and $e_{i_n}$ are elements of the canonical basis $\{e_1,\ldots,e_n\}$ of $\R^n$. Now, take $\{i_1,i_2,\ldots,i_{n-2}\}\subset \{1,2,\ldots,n\}\setminus\{i_{n-1},i_n\}$ such that $i_1<i_2<\cdots<i_{n-2}$. Then one can follow the same reasoning of the original proof of \cite[Theorem 1]{MV-S} by using $\{x_{i_1}=t_{i_1}\}$ in the place of $\{x_1=t_1\}$, $\{x_{i_2}=t_{i_2}\}$ in the place of $\{x_2=t_2\}$, and so on until $\{x_{i_{n-2}}=t_{i_{n-2}}\}$ in the place of $\{x_{n-2}=t_{n-2}\}$.
\end{proof}

In what follows we continue with the discussion of no invariance of properties by composing through the left. 
In corollaries 2.4 and 2.5 of \cite{B}, Balreira generalized to arbitrary finite dimensions the previously mentioned results on global injectivity and invertibility given by Sabatini in \cite[corollaries 2.1 and 2.3]{S} in dimension two. 

Denote by $\wedge$ the wedge product of vectors in $\R^n$ and by $\pi_j(H)$ the $j$-th homotopy group of a topological space $H$. 
Theorem 2.3 of \cite{B} is the Balreira's technical tool to prove those global injectivity and invertibility results:

\begin{balreira}\label{t:bal2.3} Let $n\geq 2,$ $f=(f_1,\ldots,f_n)\colon\R^n\to\R^n$ be a local diffeomorphism,  $k$ an integer with $1\leq k\leq n$, and $H$ an affine subspace of codimension $k$.  Assume that
	$$
	\int_{0}^{\infty}\inf_{\|x\|=r}\frac{ \|\bigwedge_{1\leq j\leq n} \nabla f_j(x) \|}{ \|\bigwedge_{\stackrel{1\leq j\leq n}{j\neq i}} \nabla f_j(x) \|}dr=\infty, \leqno(2.5)
	$$
	for each $i=1,\ldots,k$. Then $f^{-1}(H)$ is non-empty and $\pi_j(f^{-1}(H))=0$ for $j=0,1,\ldots,n-k$.  In particular, $f^{-1}(H)$ is non-empty and connected.
\end{balreira}

In the proof of Theorem 2.3 in \cite{B}, there is the following claim: 

\begin{quotation}\emph{Without lost of generality, we may also assume that the subspace $H$ is parallel to the last $n-k$ axis. Indeed, consider a rotation $A \in SO(n)$ so that $A(H)$ is parallel to the last $n-k$ axis. We then establish the result for the local diffeomorphism $A\circ f \colon \R^n \to \R^n$. In view of the above geometric interpretation of $(2.5)$ in terms of parallelepipeds, the quantity given by $(2.5)$ will remain unchanged. 
} \end{quotation}

But the following example shows that, in general, condition (2.5) is not invariant by rotations through the left, even for $n=2$. 
Since such results are not for semi-algebraic maps, we present an analytical examples for simplicity. 
But if instead of $e^x$ we take $x+\sqrt{x^2+1}$ below, we would have a semi-algebraic example. 
\begin{ex}\label{ex:1} 
	Let $f = (f_1, f_2)\colon \R^2\to\R^2$ be defined by $f({x_1},{x_2})=({x_1},e^{x_2})$.  
	Since $|\det(Df(x))|=\|\nabla f_1(x)\wedge\nabla f_2(x)\|= \|\nabla f_2(x)\|=e^{x_2}$, it follows that $f$ is a $C^\infty$ local diffeomorphism and 
	$$
	\int_{0}^{\infty}\inf_{\|x\|=r}\frac{\|\nabla f_1(x)\wedge \nabla f_2(x)\|}{ \|\nabla f_2(x)\|}\, dr=\int_{0}^{\infty}\, dr=\infty, 
	$$
	showing that $f$ satisfies (2.5) for $k=i=1$. 
	Given
	$$
	A=\frac{\sqrt{2}}{2}
	\left(
	\begin{array}{cc}
	1 & -1     \\
	1  & 1    
	\end{array}
	\right)\in SO(2), 
	$$
	define $g(x) = (g_1,g_2)(x) = (A\circ f)(x)=\sqrt{2}/2\left({x_1} - e^{x_2},{x_1} + e^{x_2}\right)$. 
	We have $|\det(Dg(x))|=\|\nabla g_1(x)\wedge\nabla g_2(x)\| = e^{x_2}$ and $\|\nabla g_i(x)\| =\sqrt{1 + e^{2x_2}}/\sqrt{2}$, which imply 
$$
\int_{0}^{\infty}\inf_{\|x\|=r}\frac{\|\nabla g_1(x)\wedge \nabla g_2(x)\|}{ \|\nabla g_i(x)\|}\, dr \leq \int_0^\infty \frac{e^{-r}}{\sqrt{1+e^{2 r}}}dr \leq 1
$$
for $i=1,2$. Hence $g$ does not satisfy (2.5). 
\end{ex}

In fact Theorem 2.3 of \cite{B} is not valid as stated, because of the following example. 

\begin{ex}\label{ex:2}
Let $q=(q_1,q_2): \R^2 \to \R^2$ be a $C^\infty$ local diffeomorphism and define the local diffeomorphism $f = (f_1, f_2, f_3): \R^3 \to \R^3$ by 
$$
f(x_1,{x_2},{x_3}) = \left({x_3}, q(x_1, x_2)\right). 
$$ 
We have $\|\nabla f_1(x)\wedge \nabla f_2(x)\wedge \nabla f_3(x)\|=\|\nabla f_2(x)\wedge \nabla f_3(x)\|= |\det Df(x)|$. 
Hence 
$$ 
\int_{0}^{\infty}\inf_{\|x\|=r}\frac{\|\nabla f_1(x)\wedge \nabla f_2(x)\wedge \nabla f_3(x)\|}{ \|\nabla f_2(x)\wedge \nabla f_3(x)\|}\, dr=\int_{0}^{\infty}\, dr=\infty, 
$$
and so condition (2.5) is satisfied for $k = i = 1$. 
Now letting $H \subset \R^3$ be the subspace of codimension $1$  spanned by $e_1=(1,0,0)$ and $e_ 2=(0,1,0)$, we have 
$$
f^{-1}(H) = q_2^{-1}(0) \times \R. 
$$ 
If $q_2^{-1}(0)$ is not connected, then $f$ is a counterexample to Theorem 2.3 of \cite{B} above. 
Take for instance $q(x) = \left(e^{x_1} x_2, e^{x_1} (1 - x_2^2)\right)$. 
\end{ex}

Removing the above statement from the proof of Theorem 2.3 of \cite{B}, the following result stays proved by Balreira in \cite{B}:

\begin{teo}\label{novo} Let $n\geq 2,$ $f=(f_1,f_2,\ldots,f_n)\colon\R^n\to\R^n$ be a local diffeomorphism,  $k$ an integer with $1\leq k\leq n$, $H$ an affine subspace of codimension $k$ \textbf{parallel to the last $n-k$ axes}.  Assume that 
$$
\int_{0}^{\infty}\inf_{\|x\|=r}\frac{ \|\bigwedge_{1\leq j\leq n} \nabla f_j(x) \|}{ \|\bigwedge_{\stackrel{1\leq j\leq n}{j\neq i}} \nabla f_j(x) \|}dr=\infty, \leqno(2.5)
$$
for each $i=1,\ldots,k$. Then $f^{-1}(H)$ is non-empty and $\pi_j(f^{-1}(H))=0$ for $j=0,1,\ldots,n-k$.  In particular, $f^{-1}(H)$ is non-empty and connected.
\end{teo}

To obtain the injectivity of $f$ in corollaries 2.4, 2.5 and 2.6 in \cite{B}, Balreira used his Theorem 2.3 to conclude that the pre-image of a line is connected, which is not proved in his paper by the remarks above. 
But Theorem \ref{novo} is sufficient to prove the injectivity of $f$ in corollaries 2.4, 2.5 and 2.6 of \cite{B}. 
In fact, when $k=n-1$, it follows from Theorem \ref{novo} that the pre-images of lines parallel to the $x_n$ axis are connected. 
This is equivalent to say that the level sets of the submersion $(f_1,f_2,\ldots,f_{n-1})\colon\R^n\to\R^{n-1}$ are connected. 
So the injectivity of $f$ follows from the fact that $f_n$ is monotone on each of such connected pre-images, which is a consequence of the local injectivity of $f$.

\end{document}